\documentclass[11pt,twoside,a4paper,reqno]{amsart}

\usepackage{amsmath}
\usepackage{amsthm}
\usepackage{amsfonts, amssymb}
\usepackage[all]{xy}
\usepackage{url}
\usepackage{algorithm}
\usepackage[end]{algpseudocode}

\usepackage{latexsym}

\usepackage{graphicx}
\usepackage{graphics}
\usepackage[scriptsize,bf]{caption}
\usepackage{float}
\usepackage{enumerate}
\usepackage{verbatim}
\usepackage{color}
\usepackage{wrapfig}
\usepackage{subfig}

\theoremstyle{plain}
\newtheorem{lema}{\sc Lemma}
\newtheorem{prop}[lema]{\sc Proposition}
\newtheorem{teo}[lema]{\sc Theorem}

\newtheorem{coro}[lema]{\sc Corollary}

\theoremstyle{definition}

\newtheorem{ej}[lema]{\sc Example}

\newtheorem{obs}[lema]{\sc Remark}

\newcommand{\less}{\lessdot}
\newcommand{\great}{\gtrdot}

\begin{document}

\title[On normalizing discrete Morse functions]{On normalizing discrete Morse functions}

\author[N. A. Capitelli]{Nicol\'as A. Capitelli\\
\textit{\scriptsize
D\MakeLowercase{epartamento de} C\MakeLowercase{iencias} B\MakeLowercase{\'asicas},  
U\MakeLowercase{niversidad} N\MakeLowercase{acional de} L\MakeLowercase{uj\'an}\\ B\MakeLowercase{uenos}
A\MakeLowercase{ires}, A\MakeLowercase{rgentina.}}}



\thanks{\textit{This work was partially supported by a Postdoc grant from CONICET}}

\thanks{\textit{E-mail address:} {\color{blue}ncapitelli@unlu.edu.ar}}

\begin{abstract} We present a way to \emph{normalize} a combinatorial Morse function into an integer-valued canonical representative of the set of discrete Morse functions inducing a given gradient field.\end{abstract}

\maketitle

\subsection*{Introduction}

In his seminal paper \cite{For}, Forman showed how to assign a combinatorial Morse function to a given (admissible) discrete vector field $W$ over a simplicial complex $K$. The value of this function on a given simplex $\sigma$ is linked to the length of gradient paths of $W$ starting in $\sigma$. In this note, we use a simplified version of Forman's construction to produce a Morse function $h:K\rightarrow\mathbb{Z}_{\geq 0}$ taking the minimum possible value for each simplex. We call one such function \emph{normalized}. Normalized Morse functions are natural representatives of all functions inducing the same gradient vector field.

Throughout, $K$ will represent a finite simplicial complex and we shall write $\sigma\prec\tau$ if $\sigma$ is an immediate face of $\tau$ in $K$ (i.e. $\dim(\sigma)=\dim(\tau)-1$). A \emph{discrete vector field} over $K$ is a map $W:K\rightarrow K\cup\{0\}$ satisfying\begin{enumerate}
\item[($W1$)] if $W(\sigma)\neq 0$ then $\sigma\prec W(\sigma)$,
\item[($W2$)] if $W(\sigma)=W(\sigma')\neq 0$ then $\sigma=\sigma'$, and
\item[($W3$)] $W^2=0$.
\end{enumerate}

The simplices in $W^{-1}(0)\setminus W(K)$ are called \emph{critical}. A \emph{combinatorial Morse function} over $K$ is a map $f:K\rightarrow\mathbb{R}$ satisfying for every $\sigma\in K$\begin{enumerate}\label{morseconditions}
\item[($M1$)] $|\{\eta\prec\sigma\,|\,f(\eta)\geq f(\sigma)\}|\leq 1$ and
\item[($M2$)] $|\{\tau\succ\sigma\,|\, f(\tau)\leq f(\sigma)\}|\leq 1$.
\end{enumerate}
Here $|X|$ denotes the cardinality of the set $X$. The \emph{gradient vector field} of a combinatorial Morse function $f$ is the discrete vector field $V_f$ over $K$ defined by $$V_f(\sigma)=\left\{\begin{array}{ll}\tau&\text{if $\tau\succ\sigma$ and $f(\sigma)\geq f(\tau)$}\\ 0&\text{otherwise.}\end{array}\right.$$
A simplex $\sigma$ is \emph{critical} for $f$ if both cardinalities in ($M1$) and ($M2$) are zero (equivalently, $\sigma$ is critical for $V_f$). Two Morse functions $f$ and $g$ are said to be \emph{equivalent} if they induce the same gradient vector field (or equivalently, $f(\sigma)<f(\tau)$ if and only if $g(\sigma)<g(\tau)$ for every $\sigma\prec\tau$; see e.g. \cite{Vil}).
Given a discrete vector field $W$ over $K$, a \emph{path of index $k$} (relative to $W$) is a sequence of $k$-simplices $\sigma_0,\ldots,\sigma_r\in W^{-1}(K)$ such that $\sigma_i\neq\sigma_{i+1}$ and $\sigma_{i+1}\prec W(\sigma_i)$ for all $0\leq i\leq r-1$. The path is \emph{closed} if $\sigma_0=\sigma_r$ and \emph{non-stationary} if $r>0$. By \cite[Theorem 9.3]{For}, $W$ is the gradient vector field of a combinatorial Morse function if and only if $W$ has no non-stationary closed paths (admissible).

\subsection*{The function underlying a gradient field} For an admissible vector field $W$ over $K$  define a relation $\less$ on $K$ by taking the transitive closure of the following \emph{atomic relations}:
\begin{enumerate}
\item $\sigma\less\tau$ if $\sigma\prec\tau$, and
\item $\sigma\great\tau$ if also $\tau=W(\sigma)$.\end{enumerate}

\noindent We shall write $\sigma\dot{\sim}\tau$ whenever $\sigma\less\tau$ and $\sigma\great\tau$. Define the \emph{height} of a simplex $\sigma$ in $(K,\less)$ as $$h(\sigma):=\max\{n\in\mathbb{Z}_{\geq 0}\,|\,\exists\,\, \sigma=\sigma_n\great \sigma_{n-1}\great\cdots\great \sigma_0,\, \sigma_i\dot{\nsim}\sigma_{i+1}\}.$$ We call $h:K\rightarrow\mathbb{Z}_{\geq 0}$ the \emph{height function associated to $W$}.

\begin{ej} The \emph{dimension function} $\dim(-):K\rightarrow\mathbb{Z}_{\geq 0}$ is the height function of the null vector field.\end{ej}

\begin{ej} For the boundary of the $2$-simplex on the vertices $\{a,b,c\}$, Figure \ref{figure} shows the height function of the vector field $W$ defined by $W(a)=ab$, $W(b)=bc$ and $W=0$ otherwise.

\begin{figure}[h]
\centering
\includegraphics[scale=0.6]{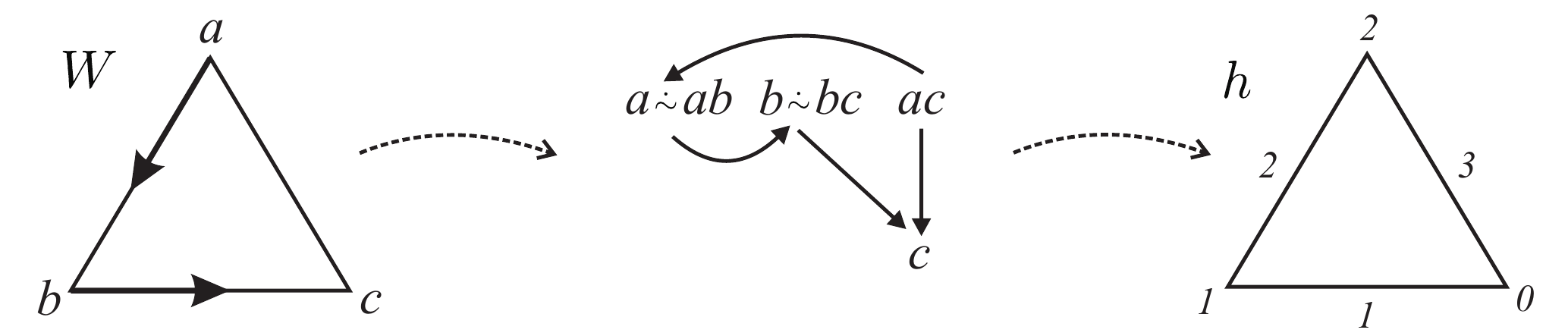}
\caption{Computing the height function $h$ for a vector field $W$.}
\label{figure}
\end{figure}\end{ej}

\begin{teo}\label{teo:ismorse} With the notations as above, $h$ is combinatorial Morse function over $K$ with gradient vector field $W$.\end{teo}

\begin{proof} First we show that if $\sigma\prec\tau$ are such that $\sigma\dot{\sim}\tau$ then $W(\sigma)=\tau$. Indeed, let $$\sigma=\eta_0\great\eta_1\great\cdots\great\eta_r=\tau$$ be a chain of atomic relations. Since $\dim(\sigma)=\dim(\tau)-1$ and $\eta_{i+1}$ is either an immediate face or coface of $\eta_i$ then $r$ must be odd and $$|\{0\leq i\leq r-1\,|\,\eta_{i+1}\prec\eta_i\}|=|\{0\leq i\leq r-1\,|\,\eta_i\prec\eta_{i+1}\}|+1.$$ Assume $r\geq 3$. Note that we cannot have $\eta_i\prec\eta_{i+1}\prec\eta_{i+2}$ for it would contradict that $W^2=0$. So we must have $\eta_i\prec\eta_{i+1}\succ\eta_{i+2}$ (with $W(\eta_i)=\eta_{i+1}$) for all odd $i$ and $\eta_i\succ\eta_{i+1}\prec\eta_{i+2}$ (with $W(\eta_{i+1})=\eta_{i+2}$) for all even $i$. This gives rise to a non-stationary closed path of index $\dim(\sigma)$, contrary to the hypotheses. We conclude that $r=1$ and thus $W(\sigma)=\tau$.

Suppose now that $\tau\succ\sigma$ satisfies $h(\tau)\leq h(\sigma)$. If $\tau\dot{\nsim}\sigma$ then adding $\tau$ to a chain that realizes the height of $\sigma$ gives us $h(\tau)\geq h(\sigma)+1$, a contradiction. So $\tau\dot{\sim}\sigma$ and $\tau=W(\sigma)$ is the only immediate co-face of $\sigma$ satisfying $h(\tau)\leq h(\sigma)$. On the other hand, if $\eta\prec\sigma$ is such that $h(\eta)\geq h(\sigma)$ then $W(\eta)=\sigma$ and $\eta$ is the only immediate face of $\sigma$ satisfying this by ($W2$). This proves that $h$ fulfills ($M1$) and ($M2$).

Finally, it is straightforward to see that $V_h=W$.\end{proof}

As mentioned earlier, the function $h$ is nothing but a simplified version of Forman's construction of the discrete Morse function produced in the proof of \cite[Theorem 9.3]{For}. Forman uses a more complex (inductive) argument in order to get some desirable extra properties for this function. Theorem \ref{teo:ismorse} provides a simpler proof to the fact that $W$ is a gradient field.

\begin{obs}\label{Obs:equivalencetrio} It is straightforward to check that if $\sigma\dot{\sim}\eta\dot{\sim}\varphi$ is an equivalence of \emph{atomic} relations ($\sigma,\varphi\neq\eta$) then necessarily $\sigma=\varphi$.\end{obs}

\subsection*{The normalization of a Morse function} Let $f:K\rightarrow\mathbb{R}$ be a discrete Morse function and let $h_f:K\rightarrow \mathbb{Z}_{\geq 0}$ be the height function associated to $V_f$. We call $h_f$ the \emph{normalization of $f$}. This function is a natural representative of the equivalence class of $f$ in the following way.

\begin{prop}\label{Prop:MinimalityConditionOfh} For every $\sigma\in K$, $h_f(\sigma)\leq g(\sigma)$ whenever $g:K\rightarrow\mathbb{Z}_{\geq 0}$ is a combinatorial Morse function equivalent to $f$.\end{prop}

\begin{proof} Suppose otherwise and let $\sigma\in\{\eta\in K\,:\,g(\eta)<h_f(\eta)\}$ for which $g(\sigma)$ is minimum (then $h_f(\sigma)=r\geq 1$). Let $C$ be a chain $\sigma=\sigma_r\great \sigma_{r-1}\great\cdots\great\sigma_0$ realizing $h_f(\sigma)$ and note that $h_f(\sigma_{r-1})=h_f(\sigma)-1$. By Remark \ref{Obs:equivalencetrio} and the maximality of $C$, any \emph{reduced} chain of atomic relations realizing $\sigma\great\sigma_{r-1}$ (i.e. not containing repeated simplices) must be of the form $\sigma\dot{\sim}\eta\great\eta'\dot{\sim}\sigma_{r-1}$, with $\eta' \dot{\nsim} \eta$. Hence $\eta'\prec\eta$ and $f(\eta')<f(\eta)$.

By the equivalence $h_f\sim f\sim g$ we have $g(\eta')<g(\eta)$. We claim that $g(\eta)\leq g(\sigma)$. By the equivalence $f\sim g$ this is straightforward if $\eta\succ\sigma$, so we must discard the case $\eta\prec\sigma$ (with $f(\sigma)\leq f(\eta)$). There are two sub-cases for this case (which arise from the possibilities for $\sigma_{r-1}\dot{\sim}\eta'$):
                                                        \begin{itemize}
                                                        \item If $\sigma_{r-1}\prec\eta'$ with $f(\sigma_{r-1})\geq f(\eta')$ then  $\sigma_{r-1}\prec\eta'\prec\eta\prec\sigma$. Let $\rho$ be the only other simplex aside $\eta'$ that fulfills $\sigma_{r-1}\prec\rho\prec\eta$. Then the chain of atomic relations $\sigma_{r-1}\less\rho\less\eta\dot{\sim}\sigma$ contradicts the maximality of $h_f(\sigma)$.
                                                        \item If $\sigma_{r-1}\succ\eta'$ with $f(\sigma_{r-1})\leq f(\eta')$ then  $\sigma_{r-1}\succ\eta'\prec\eta\prec\sigma$. Whether $\sigma_{r-1}$ is or is not an immediate face of $\sigma$, we reach the contradictions to the maximality of $h_f$ shown in Figure \ref{Democaracterizacion}.\end{itemize}

\begin{figure}[h]
  \centering
    \includegraphics[scale=0.5]{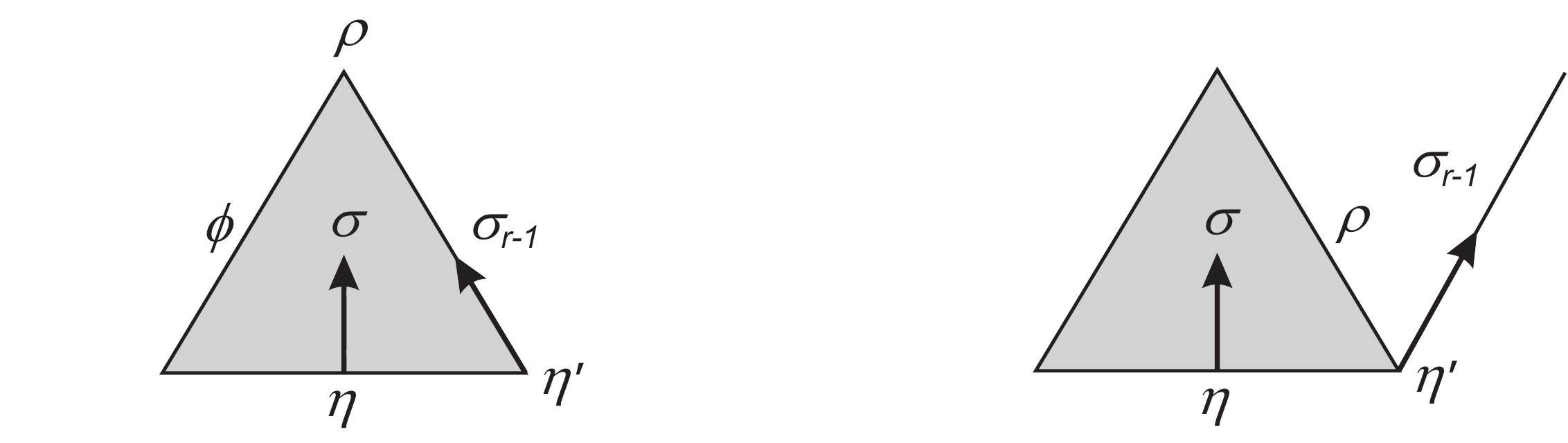}
        \caption{\textbf{On the left.} In the case $\sigma_{r-1}\prec\sigma$, the chain $\sigma_{r-1}\prec\rho\prec\phi\prec \sigma$ contradicts the maximality of $h_f(\sigma)$. \textbf{On the right.} In the case that $\sigma_{r-1}$ is not an immediate face of $\sigma$, the chain $\sigma_{r-1}\dot{\sim}\eta'\prec\rho\prec \sigma$ contradicts the maximality of $h_f(\sigma)$.}
        \label{Democaracterizacion}
\end{figure}

\noindent This proves the claim. Finally, by the minimality condition of $\sigma$ we have $$h_f(\sigma_{r-1})= h_f(\eta')\leq g(\eta')< g(\sigma)<h_f(\sigma),$$ from where $h_f(\sigma)-h_f(\sigma_{r-1})\geq 2$, a contradiction.\end{proof}

\begin{coro} If $f$ and $g$ are equivalent Morse functions then $h_f=h_g$.\end{coro}

The condition in the statement of Proposition \ref{Prop:MinimalityConditionOfh} is trivially seen to be an alternative definition for $h_f$. Thus, the normalization of $f$ is the non-negative integer-valued Morse function equivalent to $f$ taking the smallest possible values over all simplices. This justifies the term \emph{normalized}.

We collect some basic properties of $h_f$ in the following
\begin{lema} \label{Lemma:PropertiesOfh} With the notations as above, the function $h_f$ satisfies:\begin{enumerate}
\item\label{(1)} $h_f(\sigma)\geq\dim(\sigma)$ for all $\sigma\in K$.
\item\label{(2)} $h_f(\sigma)=0$ if and only if $\sigma$ is a critical vertex.
\item\label{(3)} If $\tau=V_f(\sigma)$ then $h_f(\sigma)=h_f(\tau)$.
\item\label{(4)} $h$ takes all the integer values between $0$ and $\max_{\sigma\in K}\{h_f(\sigma)\}$.
\end{enumerate}\end{lema}

\begin{proof} \eqref{(1)} and \eqref{(2)} follow from the fact that any simplex of positive dimension has at least two immediate faces, one of which must be a critical face. \eqref{(3)} is a direct consequence of the definition of the height function and \eqref{(4)} follows from Proposition \ref{Prop:MinimalityConditionOfh}.\end{proof}

\subsection*{Computing $h_f$} Finally, we present two algorithms for computing normalized functions: one from acyclic matchings over the Hasse diagram of the complex and one from the values of a given Morse function.

Recall that the \emph{Hasse diagram} $\mathcal{H}_K$ of a complex $K$ is the digraph whose vertices are the simplices of $K$ and whose directed edges $\tau\rightarrow\sigma$ arise for every $\tau\succ\sigma$. The pairing of simplices defined by a gradient vector field $V$ determines a matching $\mathcal{M}_V$ in $\mathcal{H}_K$. If the arrows in this matching are reversed in $\mathcal{H}_K$, the resulting digraph $\mathcal{H}_{\mathcal{M}}$ is acyclic. On the other hand, an \emph{acyclic matching} $\mathcal{M}$ over $\mathcal{H}_K$ determines a discrete gradient field $V_{\mathcal{M}}$ on $K$ where $V_{\mathcal{M}}(\sigma)=\tau$ if and only if the edge $\overline{\sigma\tau}\in\mathcal{M}$ (see e.g. \cite{Cha}).

\begin{algorithm}
\caption{\emph{Normalization from an acyclic matching (SAGE)}}\label{algoSAGE}
\flushleft \footnotesize In the following code, the argument \texttt{D} is the \emph{dictionary of covering vertices} of the face poset: a vertex is assigned the list of vertices covering it. A typical entry is e.g. ``\texttt{2:[3,7,13]}". Entries of maximal vertices/simplices appear e.g. as ``\texttt{30:[  ]}". On the other hand, the argument \texttt{M} is the list of paired vertices given by the acyclic matching over \texttt{D}. A typical entry in \texttt{M} is e.g. ``\texttt{[6,27]}", which is also a list.\normalsize
\begin{verbatim}


def Normalization(D,M):
    for i in range(len(M)):
        D[M[i][0]].remove(M[i][1])
        D[M[i][1]].append(M[i][0])
    G=DiGraph(D)
    V=len(G.vertices())
    for j in range(V):
        n=0
        L=G.all_simple_paths(ending_vertices=[j])
        if L!=[]:
            values=set()
            for l in L:
                k=0
                l.reverse()
                for i in range(len(l)-1):
                    if [l[i],l[i+1]] not in M:
                        k=k+1
                values.add(k)
            n=max(values)
            print 'Simplex', j, 'takes the value', n

\end{verbatim} 
\end{algorithm}

The height function associated to an acyclic matching can be characterized in terms of the number of edges in simple paths not belonging to the matching.

\begin{prop}\label{propo:hasse} Let $V$ be a gradient field over $K$, $\mathcal{M}$ the associated matching and let $P_{\sigma}$ stand for the set of (directed) simple paths in $\mathcal{H}_{\mathcal{M}}$ starting in $\sigma$. Then, $$h(\sigma)=\max_{p\in P_{\sigma}}|\{\text{edges in $p$ not in $\mathcal{M}_V$}\}|.$$\end{prop}

\begin{proof}
A simple path in $\mathcal{H}_{\mathcal{M}}$ starting in $\sigma$ is a sequence $\sigma=\sigma_n\rightarrow\sigma_{n-1}\rightarrow\cdots\rightarrow\sigma_0$ where either \begin{enumerate}\item $\sigma_{i+1}\succ\sigma_i$ and $V(\sigma_i)\neq V(\sigma_{i+1})$ or \item $\sigma_{i+1}\prec\sigma_i$ and $V(\sigma_{i+1})=\sigma_i$.\end{enumerate} Hence, one such path gives rise to a sequence $\sigma_n\great \sigma_{n-1}\great\cdots\great \sigma_0$ where $\sigma_i\dot{\sim}\sigma_{i+1}$ only for the $i$'s corresponding to simplices in the case ($2$). By removing these arrows we get a chain with no equivalencies. Since these arrows are the ``inverted arrows" in $\mathcal{M}$, the result follows.
\end{proof}

Algorithm \ref{algoSAGE} is a SAGE implementation of the algorithm implicit in Proposition \ref{propo:hasse}.

For the second algorithm we rely on the alternative characterization of $h_f$ given in Proposition \ref{Prop:MinimalityConditionOfh}. Note that one can easily scale and perturb a Morse function $f$ to get another one with the same gradient and such that $\min_{\sigma\in K}\{f(\sigma)\}=0$ and $f(K)\subset\mathbb{Z}_{\geq 0}$. Algorithm \ref{algovalues} computes the normalization of one such function.
\begin{algorithm}
\caption{\emph{Normalization from a Morse function $f$}}\label{algovalues}
\begin{algorithmic}[1]
\Require{$f:K\rightarrow\mathbb{Z}_{\geq 0}$, $\min\{f(\sigma)\}=0$.}
\Ensure{$h_f$}
\Procedure{normalize}{$f$}
%\State Let $\{a_0,\ldots,a_r\}=f(K)$ where $a_i<a_j$ if $0\leq i<j\leq r$.
\State Order $K=\{\sigma_1,\ldots,\sigma_r\}$ such that $\left\{\begin{array}{ll}
f(\sigma_i)< f(\sigma_j)& i< j\\
\dim(\sigma_i)\geq\dim(\sigma_j)&f(\sigma_i)= f(\sigma_j)\end{array}\right.$
\State $f_0\gets f$
%\State Assume $f_s$ defined for $0\leq s< r$.
\For{$0\leq s< r$}
\State $f_s^0\gets f_s$
\For{$0< k\leq r$}
\State $f_s^k(\sigma)=f_s^{k-1}(\sigma)$ if $\sigma\neq\sigma_k$
\State $f_s^k(\sigma_k)=\left\{\begin{array}{ll}
\max\{f_s^{k-1}(\eta)\,|\,\eta\prec\sigma_k\}+1&\sigma_k\text{ is critical for $f_s^{k-1}$}\\
f_s^{k-1}(\tau)& f_s^{k-1}(\sigma_k)\geq f_s^{k-1}(\tau)\\
\max\{f_s^{k-1}(\eta')\,|\,\eta'\prec\sigma_k,\,\eta'\neq\eta\}+1&f_s^{k-1}(\eta)\geq f_s^{k-1}(\sigma_k)\end{array}\right.$
\EndFor
\State $f_{s+1}\gets f^r_s$
\EndFor
\State $h_f\gets f_r$
\EndProcedure
\end{algorithmic}
\end{algorithm}

\begin{prop} Algorithm \ref{algovalues} computes the normalization of $f$.\end{prop}

\begin{proof} First we show by induction that $f_s^k$ is a Morse function with gradient $V_f$ for all $s,k$. Since the loop corresponding to a given $s$ and $k$ only alters the value of $\sigma_k$, it suffices to show that the order relations between the values of $\sigma_k$ and its immediate faces and cofaces are preserved in the transition $f_s^{k-1}\rightarrow f_s^k$ (note that this comprises the transition between subindexes $s$ as well). For simplicity, let us write $\sigma_k=\sigma$, $f_s^k=f_k$ and let $\mathcal{O}$ stand for the order the algorithm introduces on the simplices of $K$. We analyze the possible cases of $\sigma$: ($a$) $\sigma$ is critical for $f_{k-1}$, ($b$) $f_{k-1}(\sigma)\geq f_{k-1}(\tau)$ for some $\tau\succ\sigma$ and ($c$) $f_{k-1}(\eta)\geq f_{k-1}(\sigma)$ for some $\eta\prec\sigma$.
%\begin{enumerate}

In the case ($a$) a direct computation shows that for $\nu\prec\sigma\prec\rho$ we have \begin{itemize}
\item $f_k(\sigma)=\max\{f_{k-1}(\eta)\,|\,\eta\prec\sigma\}+1\leq f_{k-1}(\sigma)<f_{k-1}(\rho)=f_k(\rho)$ and
\item $f_k(\nu)=f_{k-1}(\nu)< \max\{f_{k-1}(\eta)\,|\,\eta\prec\sigma\}+1=f_k(\sigma).$ 
\end{itemize}

Suppose we are in the conditions of case ($b$). On one hand, $f_k(\sigma)=f_{k-1}(\tau)=f_k(\tau)$. On the other hand, since $f_{k-1}$ is Morse, if $\sigma\prec\rho\neq\tau$ then
$$f_k(\sigma)=f_{k-1}(\tau)\leq f_{k-1}(\sigma)<f_{k-1}(\rho)=f_k(\rho).$$ Finally, for $\nu\prec\sigma$ let $\sigma'\prec\tau$ be such that $\nu=\sigma\cap\sigma'$. Since $f_{k-1}$ is Morse then $f_{k-1}(\sigma')<f_{k-1}(\tau)$. In particular $\sigma'$ appears earlier than $\sigma$ in $\mathcal{O}$; so either $f_{k-1}(\nu)<f_{k-1}(\sigma')$ or $f_{k-1}(\nu)=f_{k-1}(\sigma')$ (if $f(\nu)\geq f(\sigma')$). In any case, $$f_k(\nu)=f_{k-1}(\nu)<f_{k-1}(\tau)=f_k(\sigma).$$

Let us now deal with case ($c$). If $\rho'\succ\sigma$ then
$$f_k(\sigma)=\max\{f_{k-1}(\eta')\,|\,\eta'\prec\sigma,\,\eta'\neq\eta\}+1<f_{k-1}(\sigma)<f_{k-1}(\rho')=f_k(\rho').$$ On the other hand, if $\sigma\succ\nu\neq\eta$ then 
$$f_k(\sigma)=\max\{f_{k-1}(\eta')\,|\,\eta'\prec\sigma,\,\eta'\neq\eta\}+1>f_{k-1}(\nu)=f_k(\nu).$$ Finally, using that $f_{k-1}$ is Morse, $$f_k(\sigma)=\max\{f_{k-1}(\eta')\,|\,\eta'\prec\sigma,\,\eta'\neq\eta\}+1\leq f_{k-1}(\sigma)\leq f_{k-1}(\eta)=f_k(\eta).$$
%\end{enumerate}
This proves that $f_r$ is a Morse function with the same gradient field as $f$. In order to show that $f_r=h_f$ we prove the minimality of $f_r$ according to Proposition \ref{Prop:MinimalityConditionOfh}. Assume there exists a Morse function $g:K\rightarrow\mathbb{Z}_{\geq 0}$ with $V_g=V_{f_r}$ contradicting this minimality and let $\sigma=\sigma_k$ be the earliest simplex in the order $\mathcal{O}$ such that $g(\sigma)<f_r(\sigma)$. If $\sigma$ is in case ($a$) then all the immediate faces of $\sigma$ appear earlier in $\mathcal{O}$ so $$\begin{array}{rcl}
g(\sigma)<f_r(\sigma)=f_k(\sigma)&=&\max\{f_{k-1}(\eta)\,|\,\eta\prec\sigma\}+1\\
&=&\max\{f_r(\eta)\,|\,\eta\prec\sigma\}+1\\
&\leq& \max\{g(\eta)\,|\,\eta\prec\sigma\}+1.\end{array}$$ Thus if $\nu\prec\sigma$ realizes this last maximum then $g(\sigma)\leq g(\nu)$, contrary to the hypothesis of ($a$). If $\sigma$ is in the conditions of ($b$) then $\tau$ appears earlier than $\sigma$ in $\mathcal{O}$ and we reach the following contradiction:
$$f_r(\tau)\leq g(\tau)\leq g(\sigma)<f_r(\sigma)=f_k(\sigma)=f_{k-1}(\tau)=f_r(\tau).$$
Finally, if we are in case ($c$), let $\nu\prec\sigma$ be such that $f_k(\nu)=\max\{f_{k-1}(\eta')\,|\,\eta'\prec\sigma,\,\eta'\neq\eta\}$. Since $\nu$ appears earlier than $\sigma$ in $\mathcal{O}$ then $$g(\nu)<g(\sigma)<f_r(\sigma)=f_k(\sigma_k)=f_{k-1}(\nu)+1=f_r(\nu)+1,$$ which contradicts the choice of $\sigma$ as a minimum.\end{proof}

\end{document}